\documentclass[12pt]{amsart}
\usepackage{amsmath, latexsym, amsfonts, stmaryrd, amssymb, amsthm, amscd, array, caption, enumerate, array}
\usepackage[T1]{fontenc}
\usepackage[toc,page]{appendix}
\usepackage[utf8]{inputenc}
\usepackage{graphics,epsf,psfrag,epsfig}
\usepackage{color}

\numberwithin{equation}{section}
\newtheorem{theorem}{Theorem}[section]
\newtheorem{lemma}[theorem]{Lemma}
\newtheorem{proposition}[theorem]{Proposition}

\newtheorem{rem}[theorem]{Remark}

\renewenvironment{proof}[1][Proof]{\begin{trivlist}
\item[\hskip \labelsep {\bfseries #1}]}{\qed\end{trivlist}}

\newcommand{\ind}{\mathbf{1}}
\renewcommand{\ge}{\geq}

\newcommand{\Z}{\mathbb{Z}}

\renewcommand{\tilde}{\widetilde}
\renewcommand{\hat}{\widehat}

\DeclareMathSymbol{\leqslant}{\mathalpha}{AMSa}{"36} 
\DeclareMathSymbol{\geqslant}{\mathalpha}{AMSa}{"3E} 
\DeclareMathSymbol{\eset}{\mathalpha}{AMSb}{"3F}     
\renewcommand{\leq}{\;\leqslant\;}                   
\renewcommand{\geq}{\;\geqslant\;}                   

\newcommand{\mA}{{\ensuremath{\mathcal A}} }

\newcommand{\mL}{{\ensuremath{\mathcal L}} }

\newcommand{\mD}{{\ensuremath{\mathcal D}} }

\newcommand{\bP}{{\ensuremath{\mathbf P}} }
\newcommand{\bE}{{\ensuremath{\mathbf E}} }

\newcommand{\ga}{\alpha}

\newcommand{\gd}{\delta}
\newcommand{\gep}{\varepsilon}       
\newcommand{\gp}{\varphi}
\newcommand{\tgp}{\tilde\gp}

\newcommand{\gl}{\lambda}

\makeatletter
\def\captionfont@{\footnotesize}
\def\captionheadfont@{\scshape}

\long\def\@makecaption#1#2{%
  \vspace{2mm}
  \setbox\@tempboxa\vbox{\color@setgroup
    \advance\hsize-6pc\noindent
    \captionfont@\captionheadfont@#1\@xp\@ifnotempty\@xp
        {\@cdr#2\@nil}{.\captionfont@\upshape\enspace#2}%
    \unskip\kern-6pc\par
    \global\setbox\@ne\lastbox\color@endgroup}%
  \ifhbox\@ne 
    \setbox\@ne\hbox{\unhbox\@ne\unskip\unskip\unpenalty\unkern}%
  \fi
  \ifdim\wd\@tempboxa=\z@ 
    \setbox\@ne\hbox to\columnwidth{\hss\kern-6pc\box\@ne\hss}%
  \else 
    \setbox\@ne\vbox{\unvbox\@tempboxa\parskip\z@skip
        \noindent\unhbox\@ne\advance\hsize-6pc\par}%
\fi
  \ifnum\@tempcnta<64 
    \addvspace\abovecaptionskip
    \moveright 3pc\box\@ne
  \else 
    \moveright 3pc\box\@ne
    \nobreak
    \vskip\belowcaptionskip
  \fi
\relax
}
\makeatother
\def\writefig#1 #2 #3 {\rlap{\kern #1 truecm
\raise #2 truecm \hbox{#3}}}

\newcommand{\tga}{\tilde{\alpha}}

\renewcommand{\subset}{\subseteq}

\renewcommand{\phi}{\varphi}
\newcommand{\bPs}{\bP_{\sigma}}

\newcommand{\bPt}{\bP_{\tau} }

\newcommand{\simn}{\stackrel{n\to\infty}{\sim}}

\begin{document}

\title[Intersection of two independent renewals]{Local asymptotics for the first intersection of two independent renewals} 

\author[K. Alexander]{Kenneth S. Alexander}
\address{Department of Mathematics, KAP 108\\
University of Southern California\\
Los Angeles, CA  90089-2532 USA}

\author[Q. Berger]{Quentin Berger}
\address{LPMA, Universit\'e Pierre et Marie Curie\\
Campus Jussieu, case 188\\
4 place Jussieu, 75252 Paris Cedex 5, France}
\email{quentin.berger@upmc.fr}

\begin{abstract}
We study the intersection of two independent renewal processes, $\rho=\tau\cap\sigma$.
Assuming that $\bP(\tau_1 = n ) = \gp(n)\, n^{-(1+\ga)}$ and $\bP(\sigma_1 = n ) = \tilde\gp(n)\, n^{-(1+ \tilde\ga)} $ for some $\ga,\tga \geq 0$ and some slowly varying $\gp,\tilde\gp$, we give the asymptotic behavior first of $\bP(\rho_1>n)$ (which is straightforward except in the case of $\min(\ga,\tilde\ga)=1$) and then of $\bP(\rho_1=n)$.  
The result may be viewed as a kind of reverse renewal theorem, as we determine probabilities $\bP(\rho_1=n)$ while knowing asymptotically the renewal mass function $\bP(n\in\rho)=\bP(n\in\tau)\bP(n\in\sigma)$.
Our results can be used to bound coupling-related quantities, specifically the increments $|\bP(n\in\tau)-\bP(n-1\in\tau)|$ of the renewal mass function.
\end{abstract}

\maketitle

\section{Intersection of two independent renewals}

We consider two independent (discrete) renewal processes $\tau$ and $\sigma$, whose law are denoted respectively $\bPt$ and $\bPs$, and the renewal process of  intersections, $\rho=\tau\cap\sigma$. We denote $\bP = \bP_{\tau} \otimes \bP_{\sigma}$.

\smallskip
The process $\rho$ appears in various contexts.  In pinning models, for example, it may appear directly in the definition of the model (as in \cite{cf:ABpinning}, where $\sigma$ represents sites with nonzero disorder values, and $\tau$ corresponds to the polymer being pinned) or it appears in the computation of the variance of the partition function via a replica method (see for example \cite{cf:Tonin}), and is central in deciding whether disorder is relevant or irrelevant in these models, cf. \cite{cf:BL}.

\smallskip
When $\tau$ and $\sigma$ have the same inter-arrival distribution, $\rho_1$ is related to the coupling time of $\tau$ and $\sigma$, if we allow $\tau$ and $\sigma$ to start at different points.
In particular, in the case $\mu:=\bE[\tau_1]<+\infty$, the coupling time $\rho_1$ has been used to study the rate of convergence in the renewal theorem, see \cite{cf:Lindvall, cf:Ney}, using that
\[|\bP(n\in \tau) - \bP(n\in\sigma)| \leq \bE\big[  |\ind_{\{n\in\tau\}} -\ind_{\{n\in\sigma\}}| \ind_{\{\rho_1 >n\}} \big] \leq \bP(\rho_1 >n) \, .\]
Hence, if $\sigma$ is \emph{delayed} by a random $X$ having the waiting time distribution $\nu$ of the renewal process (and denoting $\bP_{\nu}$ the delayed law of $\sigma$), we have that $\bP_{\nu}(n\in\sigma) =\frac{1}{\mu}$ for all $n$, and so $\bP_{\tau} \otimes \bP_{\nu}(\rho_1 >n )$ gives the rate of convergence in the renewal theorem.
This question has also been studied via a more analytic method in \cite{cf:Rogozin,cf:Grubel}.
Denoting $u_n:=\bP(n\in\tau)$ the renewal mass function ot $\tau$, Rogozin \cite{cf:Rogozin} proved that $u_n-\tfrac1\mu \sim \frac{1}{\mu^2} \sum_{k>n}^{+\infty} \bP(\tau_1>k)$ as $n\to\infty.$

In this paper, we consider only the non-delayed case, with a brief exception to study $|u_n-u_{n-1}|$, see  Theorem \ref{thm:un}.

\subsection{Setting of the paper}

We assume that there exist $\ga,\tga\geq 0$ and slowly varying functions $\gp,\tgp$ such that
\begin{equation}
\label{def:alphas}
\bP(\tau_1 = n ) = \gp(n)\, n^{-(1+\ga)}\, , \qquad \bP(\sigma_1 = n ) = \tilde\gp(n)\, n^{-(1+ \tilde\ga)} \, .
\end{equation}
(As mentioned above, $\tau$ and $\sigma$ are \emph{non-delayed}, if not specified otherwise.)
With no loss of generality, we assume that $\ga\leq \tga$.  We define $\mu_n:=\bE[\tau_1\wedge n]$ and $\tilde \mu_n:= \bE[\sigma_1 \wedge n]$ the truncated means, and also $\bE[\tau_1] = \mu= \lim_{n\to\infty} \mu_n \leq \infty$, and similarly $\tilde \mu =\lim_{n\to\infty} \tilde \mu_n$.

The assumption \eqref{def:alphas} is very natural, and is widely used in the literature (for example, once again in pinning models).
It covers in particular the case of the return times $\tau=\{n\, , \, S_{2n}=0\}$, where $(S_n)_{n\geq 0}$ is the simple symmetric nearest-neighbor random walk on $\Z^d$ (see e.g.\ \cite[Ch.~III]{cf:Feller} for $d=1$, \cite[Thm.\ 4]{cf:JainPruitt} for $d=2$ and \cite[Thm.\ 4]{cf:DK} for $d=3$), or the case $\tau=\{n\, , S_n=0\}$ where $(S_n)_{n\geq 0}$ an aperiodic random walk in the domain of attraction of a symmetric stable law, see \cite[Thm.~8]{cf:Kesten}.

\medskip
In Section \ref{sec:renewal}, we recall the strong renewal theorems for $\tau$ and $\sigma$ under assumption \eqref{def:alphas} (from \cite{cf:ABzero,cf:Doney,cf:Eric} in the recurrent case, \cite[App.~A.5]{cf:Giac} in the transient case), as well as newer reverse renewal theorems (from \cite{cf:ABzero}).
We collect the results when $\tau$ is recurrent in the following table, denoting $r_n:=\bP(\tau_1>n)$, and we refer to \eqref{eq:transient} for the transient case.
\begin{table}[htbp]\centering
\begin{tabular}{ |c|c|c|c|}
\hline
  & $\ga\geq1$ & $\ga\in(0,1)$ & $\ga=0$
\\
\hline
$\bP(n\in\tau) \simn$        &     $ \big( \mu_n\big)^{-1}$      &  $\frac{\ga \sin(\pi \ga)}{\pi}\ n^{-(1-\ga)} \gp(n)^{-1}$     &  $\frac{\gp(n)}{n\,  r_n^{2}} $
\\
\hline
\end{tabular}
\caption{ Asymptotics of the renewal mass function if $\tau$ is recurrent, and has  inter-arrival distribution $\bP(\tau_1=n)=\gp(n) n^{-(1+\ga)}$ with $\ga\geq 0$.
} \label{table:Doney} 
\end{table}

From Table \ref{table:Doney} and \eqref{eq:transient}, the renewal mass  function of $\rho$ satisfies
\begin{equation}
\label{Prho}
\bP(n\in\rho) = \bP(n\in\tau)\bP(n\in\sigma) =  \psi^*(n) \, n^{-\theta^*}
\end{equation}
for some $\theta^* \geq 0$ and slowly varying function $\psi^*(n)$.
For example, if both $\tau$ and $\sigma$ are recurrent we have 
\begin{equation}\label{thetastar}
  \theta^* = 2-\ga\wedge 1-\tga\wedge 1;
  \end{equation}
if also $\ga,\tilde\ga \in(0,1)$, then $\psi^*$ is a constant multiple of $1/\gp\tgp$. If instead both $\tau$ and $\sigma$ are transient then $\theta^*=2+\alpha+\tilde \alpha$. Note that $\rho$ is transient for $\theta^*>1$ and recurrent for $\theta^*<1$. Recalling that $\ga\leq\tga$, if we define
\begin{equation}
\label{def:alpha*}
  \alpha^* = \begin{cases} \alpha &\text{if $\rho$ is recurrent and } \alpha \geq 1,\\
    1-\theta^* &\text{if $\rho$ is recurrent and } \alpha < 1,\\
    \theta^*-1 &\text{if $\rho$ is transient}, \end{cases}
\end{equation}
then, based on Theorem \ref{thm:transient} in the transient case and 
Table \ref{table:Doney} in the recurrent case, we expect $\bP(\rho_1=n)$ to be expressed as $n^{-(1+\alpha^*)}$ multiplied by a slowly varying function.

Observe that the renewal function of $\rho$, defined as 
\[
U_n^*:=\sum_{k=0}^{n} \bP(n\in\rho),
\]
is always regularly varying, with exponent $\alpha^*=1-\theta^*$ in the recurrent case and $0$ in the transient case.

\medskip

Our goal is to derive from \eqref{def:alphas} the local asymptotics of the inter-arrival distribution, that is, the asymptotics of $\bP(\rho_1 =n)$.  For general renewal processes $\rho$ these asymptotics should not be uniquely determined by the asymptotic behavior of the renewal mass function \eqref{Prho} (which is known is our case), but the extra structure given by $\rho=\tau\cap\sigma$ under \eqref{def:alphas} makes such determination possible.

\begin{rem}\rm
\label{rem:rhorecurrent}
For $\rho$ to be recurrent, it is necessary that both $\tau$ and $\sigma$ are recurrent, so \eqref{thetastar} holds.
It follows from Table \ref{table:Doney} that $\rho$ is recurrent if and only if one of the following also holds:
\begin{itemize}\baselineskip 10pt 
\item[(i)] $\ga+\tga > 1$, \\
\item[(ii)] $\ga,\tga \in(0,1),\ga+\tga=1$ and $\sum_{n\geq 1} \frac{1}{n \gp(n)\tgp(n)} =+\infty$,\\
\item[(iii)] $\ga=0,\tga=1$ and $\sum_{n\geq 1} \frac{\gp(n)}{n \, r_n^2 \tilde \mu_n} =+\infty$.
\end{itemize}
\end{rem}

\subsection{Main results}

\subsubsection*{Case of transient $\rho$}

Since $\bP(n\in\rho)$ is summable (with sum $\bE(|\rho|)$), we must have $\theta^*\geq 1$. Here the following is immediate from (\cite{cf:ABzero}, Theorem~1.4), given below as Theorem \ref{thm:transient}.

\begin{theorem}
Assume \eqref{def:alphas}, and suppose that $\rho$ is transient. Then 
\[\bP(\rho_1 = n) \stackrel{n\to\infty}{\sim} \frac{1}{\bE(|\rho|)^2} \bP(n\in\tau)  \bP(n\in\sigma) \, .\]
\end{theorem}

\subsubsection*{Case of recurrent $\rho$}

Here $\tau$ and $\sigma$ must be recurrent, so \eqref{thetastar} holds with $\theta^*\in [0,1]$, and $\ga^* = 1- \theta^*$ if $\ga\leq 1$, $\ga^*= \ga$ if $\ga \geq 1$.

\begin{theorem}
\label{thm:Prhon2}
Assume \eqref{def:alphas}, and suppose that $\rho$ is recurrent. Then for $\alpha^*$ from \eqref{def:alpha*} (with $\theta^*$ defined as in \eqref{thetastar}), the following hold.

(i) If $\alpha^* \in(0,1)$ then
\begin{equation}
\label{Prhon1}
\bP(\rho_1 > n) \simn \frac{\sin(\pi\alpha^*)}{\pi} \psi^*(n)^{-1}\, n^{-\alpha^*}\, .
\end{equation}

(ii) If $\alpha^*=0$ then
\begin{equation}
\label{Prhon4}
  \bP(\rho_1 > n) \simn \left(  \sum_{j=1}^n \frac{\psi^*(j)}{j} \right)^{-1},
\end{equation}
which is slowly varying.

(iii) If $\alpha^*\geq 1$ then
\begin{equation}
\label{Prhon2}
\bP(\rho_1 > n) \simn \tilde\mu_n\bP(\tau_1>n) + \mu_n\bP(\sigma_1>n).
\end{equation}
\end{theorem}

In Theorem \ref{thm:Prhon2}(iii), $\tilde\mu_n,\mu_n$ are slowly varying since $\tga\ge\ga\geq 1$ (recall \eqref{def:alpha*}), and they may be replaced by $\mu$ or $\tilde\mu$ if that mean is finite.

\smallskip
We will prove Theorem \ref{thm:Prhon2} in Sections \ref{sec:alpha*pos}--\ref{sec:alpha*0inf}. 
The cases (i) and (ii) are essentially immediate from known relations of form $\bP(\rho_1>n)\sim c/U_n^*$ and are given in Section~\ref{sec:alpha*pos}. Item (iii) seems to be a new result, and is treated in Section \ref{sec:alpha*0inf} via a probabilistic method.
Note that in all cases, $\bP(\rho_1>n)$ is regularly varying with exponent $-\ga^*$.

\medskip

To obtain the asymptotics of $\bP(\rho_1=n)$ from Theorem~\ref{thm:Prhon2} (in the case $\alpha^*>0 $), or using the weak reverse renewal Theorem~\ref{thm:AB} (in the case $\ga^*=0$),  we only need to show that $\bP(\rho_1=k)$ is approximately constant on an interval $[(1-\gep)n,n]$ with $\gep$ small. To that end we have the following lemma, which we will prove in Section~\ref{sec:regvary}.

\begin{lemma}
\label{lem:regvarying}
Assume \eqref{def:alphas}, and suppose that $\rho$ is recurrent. Let $v_n:= \bP(\rho_1>  n)^2 \bP(n\in\rho)$.
Then for every $\gd>0$, there exists some $\gep>0$ such that, if $n$ is large enough we have for all $k\in(0,\gep n)$
\begin{equation}\label{stretched}
  (1-\gd) \bP(\rho_1=n-k)-\gd v_n \leq \bP(\rho_1=n) \leq (1+\gd) \bP(\rho_1=n+k) + \gd v_n.
\end{equation}
\end{lemma}

We will see later that $v_n = O(\bP(\rho_1=n))$, so Lemma \ref{lem:regvarying} is actually true without the $\gd v_n$ terms, but we will not need this improved result.

\smallskip
We can now state our main theorem, which we will prove in Section \ref{sec:regvarying}. 

\begin{theorem}
\label{thm:main2}
Assume \eqref{def:alphas} with $\tga\geq \ga$, and suppose that $\rho$ is recurrent. Let $\ga^*$ be as in \eqref{def:alpha*}.

(i) If  $\alpha^* \in(0,1)$ then
\[
 \bP(\rho_1=n) \stackrel{n\to\infty}{\sim} \frac{ \alpha^*\, \sin(\pi \alpha^*) }{\pi }\,   \psi^*(n)^{-1} n^{-(1+\alpha^*)}  \, .
 \]

(ii) If $\alpha^*=0$ then
\[
  \bP(\rho_1=n) \simn \left( \sum_{k=1}^n \frac{\psi^*(k)}{k} \right)^{-2}  \frac{\psi^*(n)}{n}\, . 
  \]

(iii) If $\alpha^*\geq 1$ then
\[
  \bP(\rho_1=n)  \simn \tilde\mu_n \, \bP(\tau_1 = n) + \mu_n\,  \bP(\sigma_1= n) \simn \frac{\tilde\mu_n \gp(n) + \mu_n \tgp(n) n^{\ga-\tga}}{n^{1+\ga^*}} \, .
  \]
\end{theorem}

\medskip
As in Theorem \ref{thm:Prhon2}, in (iii), $\tilde\mu_n,\mu_n$ may be replaced by $\mu$ or $\tilde\mu$ if that mean is finite.

We now illustrate this theorem with some subcases, using Table~\ref{table:Doney}.
\smallskip

{\bf 1.} If $\tau,\sigma$ are recurrent with $\ga,\tga\in(0,1)$ and $\ga+\tga>1$, then $\rho$ is recurrent with $\alpha^*=\ga+\tga-1\in(0,1)$ and
\begin{equation}
\label{gatga01}
 \bP(\rho_1=n) \stackrel{n\to\infty}{\sim} c_{\ga,\tga}\,    \gp(n) \tilde\gp(n) \, n^{-(\ga+\tga)} \qquad \text{with } c_{\ga,\tga} = \frac{\pi\alpha^* \sin(\pi\alpha^*)}{\ga\tga \sin(\pi \ga) \sin(\pi \tga)} \, .
\end{equation}

\smallskip
{\bf 2.} If $\tau,\sigma$ are recurrent with $\ga,\tga\in(0,1), \ga+\tga=1$ and $\sum_{n=1}^\infty 1/n\gp(n)\tgp(n)=\infty$, then $\rho$ is recurrent, $\alpha^*=0$, $\psi^*(n) \sim c'_{\ga,\tga} \gp(n)^{-1} \tgp(n)^{-1}$ with $c'_{\ga,\tga} = \frac{\ga\tga \sin(\pi \ga) \sin(\pi\tga)}{\pi^2} $. Therefore,
\begin{equation}
\bP(\rho_1=n) \simn \frac{1}{c'_{\ga,\tga}} \left(\sum_{k=1}^n \frac{1}{ k \gp(k)\tilde\gp(k)}\right)^{-2}  \frac{1}{n \gp(n)\tilde\gp(n)} \,  .
\end{equation}
As a special case, suppose $\tau = \{n, S_{2n} = 0\}$, $\sigma= \{n, S'_{2n}=0\}$ are the return times of independent symmetric simple random walks (SSRW) on $\mathbb{Z}$.  Then $\ga=\tga=1/2$ and $\gp(n)=\tgp(n)\to \frac{1}{2\sqrt{\pi}}$ so 
\begin{equation}\label{jain}
  \bP(\rho_1=n) \simn \frac{\pi}{n(\log n)^2}.
\end{equation}
Rotating the lattice by $\pi/4$ shows that this is the same as the return time distribution for $(\mathbf{S}_n)_{n\geq 0}$ the SSRW on $\mathbb{Z}^2$ (the \emph{even} return times: $\rho=\{n,\mathbf{S}_{2n} =0\}$). Hence \eqref{jain} is a classical result of Jain and Pruitt \cite{cf:JainPruitt}.

\smallskip
{\bf 3.}
If $\tau$ is recurrent with $\ga\in(0,1)$, and $\tga\geq 1$ (so $\tilde \mu_n$ is slowly varying; this includes the case when $\tilde\mu<+\infty$), then $\alpha^*=\ga$ and 
\begin{equation}
\label{ga01tga>1}
\bP(\rho_1 =n ) \stackrel{n\to\infty}{\sim} \tilde{\mu}_n\, \gp(n) n^{-(1+\ga)}  \stackrel{n\to\infty}{\sim} \tilde{\mu}_n\, \bP(\tau_1=n)  \, .
\end{equation}

\subsection{Application to a coupling-related quantity}
\label{sec:coupling}

We now provide an application of Theorem~\ref{thm:Prhon2}.
\begin{theorem}
\label{thm:un}
Let $\tau$ be a recurrent renewal process satisfying \eqref{def:alphas}, and let $u_n:=\bP(n\in\tau)$ be its renewal mass function. There exist constants $c_i>0$ such that
\begin{equation}\label{deltaun}
  | u_n - u_{n-1} |  \leq c_1\, u_n \bP(\rho_1>n) \leq \begin{cases} c_2\, n^{-1/2}\gp(n)^{-1}\left( \sum_{j=1}^n \frac{1}{j\gp(j)^2} \right)^{-1} 
  &\text{if } \ga=1/2\\ c_2\, n^{-\ga}\gp(n)^{-1} &\text{if } \ga>1/2. \end{cases}
  \end{equation}
\end{theorem}

\smallskip
Note that the right side of \eqref{deltaun} is of order $\bP(\tau_1>n)$ when $\ga>1/2$. It is summable precisely when $\mu=\bE[\tau_1]<+\infty$, and then, by Theorem \ref{thm:Prhon2}(iii), \eqref{deltaun} says $|u_n-u_{n-1}| \leq c_3\ \bP(\tau_1 >n)$. This gives additional information compared to the known asymptotics 
\[
  u_n-\frac1\mu \sim \frac{1}{\mu^2} \sum_{k>n} \bP(\tau_1>k)
  \]
from \cite{cf:Rogozin}.  We can sum \eqref{deltaun} to obtain $|u_n - 1/\mu| \leq c_3 \sum_{k>n} \bP(\tau_1>k)$, which is of the right order, but we cannot obtain the proper constant $1/\mu^2$.

We also mention the works of Topchii \cite{cf:Top1,cf:Top2}, treating the case when $\tau_1$ is a continuous random variable with $E[\tau_1]=\infty$ and density $f(t)\stackrel{t\to\infty}{\sim} \gp(t) t^{-(1+\ga)}$, studying the density $u(t)=\sum_{k=0}^{\infty} f^{*k}(t)$ $u(t)$ of the renewal function, and also $u'(t)$. Under some additional regularity conditions on $f'(t)$, letting $m(t):=\bE[\tau_1\wedge t]$, it is proven that
\[u'(t) \stackrel{t\to\infty}{\sim} 
\begin{cases}
 \frac{\ga (\ga-1) \sin(\pi \ga)}{\pi} \  \gp(t)^{-1}\, t^{-(2-\ga)} & \quad \text{ if } 0< \ga<1 ,\\
  \frac{1}{m(t)^2}\ \gp(t)\, t^{-1}  &  \quad \text{ if $\ga=1$ and } \bE[\tau_1]=\infty.
\end{cases}
\]
This is a better estimate than its analog in the infinite-mean case in Theorem \ref{thm:un}, but the techniques of \cite{cf:Top1,cf:Top2} do not appear adaptable to the discrete setting.

\begin{proof}[Proof of Theorem \ref{thm:un}.]
The second inequality in \eqref{deltaun} is a direct consequence of Theorem \ref{thm:Prhon2}(iii) and Table \ref{table:Doney}, so we prove the first one.
Take $\sigma$ a renewal process independent from $\tau$, with the same inter-arrival distribution, but starting from $\sigma_0=1$. We can couple $\tau$ and $\sigma$ so that $\tau=\sigma$ on $[\rho_1,\infty)$. Then denoting the corresponding joint distribution by $\bP_{0,1}$ we have
\begin{align*}
| u_n - u_{n-1} | &= \big| \bE_{0,1}[ \ind_{\{n\in\tau\}} - \ind_{\{n\in\sigma\}}] \big| \leq \bP_{0,1}(n\in\tau,\rho_1>n) +\bP_{0,1}(n\in \sigma,\rho_1>n).
\end{align*}
By Lemma \ref{lem:appendix} there is a constant $C_0$ such that
\[
  \bP_{0,1}(n\in\tau,\rho_1>n) \leq \bP_{0,1}(n\in\tau) \bP_{0,1}(\rho_1 >n/4 \mid n\in\tau) \leq C_0\, u_n \bP_{0,1}(\rho_1>n/4) ,
  \]
and similarly for $\bP_{0,1}(n\in \sigma,\rho_1>n)$, since $u_{n-1} \sim u_n$. Now, fix $k_0$ such that $\bP(\tau_1=k_0+1)\bP(\tau_1=k_0) >0$, and observe that for any $x>0$
\[
  \bP(\rho_1> x+k_0) \geq \bP(\sigma_1=k_0+1)\bP(\tau_1=k_0) \bP_{0,1}(\rho_1 > x).
\]
Since $\bP(\rho_1>n)$ is regularly varying (cf.~Theorem \ref{thm:Prhon2}), it follows that there is a constant $c_4 >0$ such that $\bP_{0,1}(\rho_1>n/4) \leq c_4 \bP(\rho_1 >n)$, and hence Theorem \ref{thm:un} follows.
\end{proof}

\subsection{Organization of the rest of the paper and idea of the proof}

First of all, we recall renewal and reverse renewal theorems in Section \ref{sec:renewal}, which are used throughout the paper.

\smallskip
Sections \ref{sec:alpha*pos}--\ref{sec:alpha*0inf} are devoted to the proof of Theorem \ref{thm:Prhon2}. Items (i)-(ii) are dealt with using Theorem 8.7.3 in \cite{cf:BGT}, and our main contribution is the proof of item (iii). 
The underlying idea is that, in order to have $\{\rho_1>n\}$ either one of $\tau$ or $\sigma$ typically makes a jump of order at least $n$. We decompose $\bP(\rho_1>n)$ according to the number $k$ of steps before $\tau$ (resp. $\sigma$) \emph{escapes} beyond $n$ by a jump larger than $(1-\gep)n$: we find that the expected number of steps is approximately $\tilde \mu_n$ (resp. $\mu_n$), giving Theorem \ref{thm:Prhon2}(iii).

\smallskip
Sections \ref{sec:regvary}--\ref{sec:regvarying} contain the proof of Theorem \ref{thm:main2}. In Section \ref{sec:regvary}, we prove Lemma~\ref{lem:regvarying} in two steps. First, we show that when $\rho_1=n$, having only gaps of length $\leq \gd n$ is very unlikely ; then, given that there is, say in $\tau$, a gap larger than $\gd n$, we can stretch it (together with associated $\sigma$ intervals) by $k \ll \gd n$ at little cost: this proves that $\bP(\rho_1 = n) \approx \bP(\rho_1=n+k)$.
In Section \ref{sec:regvarying}, we conclude the proof of Theorem \ref{thm:main2} by combining Lemma \ref{lem:regvarying} with Theorem \ref{thm:Prhon2}.

\section{Background on renewal and reverse renewal theorems}
\label{sec:renewal}

We consider a renewal $\tau=\{\tau_0,\tau_1,\dots\}$, with $\tau_0=0$. 
The corresponding \emph{renewal mass function} is $\bP(n\in\tau), n\geq 0$.

\subsection{On renewal theorems}

In what follows we assume that the inter-arrival distribution of $\tau$ satisfies \eqref{def:alphas}.

\subsubsection*{Transient case}
If $\tau$ is transient, then (see \cite[App. A.5]{cf:Giac}) 
\begin{equation}
\label{eq:transient}
\bP(n\in\tau) \simn \frac{1}{(p_\infty^\tau)^2} \bP(\tau_1=n) \, ,
\end{equation}
where $p_\infty^\tau:= \bP(\tau_1=+\infty)\in(0,1)$.

\subsubsection*{Recurrent case}
Here there are multiple subcases, as follows.

\smallskip

\textbullet\ If $\bE[\tau_1]<+\infty$, then the classical Renewal Theorem says
\begin{equation}
\lim_{n\to\infty}  \bP(n\in\tau) =  \frac{1}{\bE[\tau_1]}.
\end{equation}

\smallskip

\textbullet\ If $\ga=1$, $\bE[\tau_1]=+\infty$, then from \cite[eq. (2.4)]{cf:Eric},
\begin{equation}
\label{Doney1}
\bP(n\in\tau) \stackrel{n\to\infty}{\sim} (\mu_n)^{-1}\, ,
\end{equation}
where $\mu_n:= \bE(\tau_1\wedge n) $ is slowly varying.

\smallskip
\textbullet\ If $\ga\in(0,1)$ then by \cite[Thm. B]{cf:Doney}, 
\begin{equation}
\label{Doney}
\bP(n\in\tau) \stackrel{n\to\infty}{\sim} \frac{\ga \sin(\pi \ga)}{\pi}\ n^{-(1-\ga)} \gp(n)^{-1}\, .
\end{equation}
(Note that there is a typo in \cite[Eq. (1.8)]{cf:Doney}.)

\smallskip
\textbullet\ If $\ga=0$, then from \cite[Thm. 1.2]{cf:ABzero}, 
\begin{equation}
\label{alpha0}
\bP(n\in\tau) \stackrel{n\to\infty}{\sim} \frac{\bP(\tau_1=n)}{ \bP(\tau_1>n)^2}\, .
\end{equation}

We recall that the results in the case of a recurrent $\tau$ are collected in Table~\ref{table:Doney}.

\subsection{On reverse renewal theorems}

In the opposite direction, if in place of \eqref{def:alphas}, one assumes that $\bP(n\in\tau)$ is regularly varying with exponent $1-\ga$, then for $0\leq \alpha<1$ the asymptotics of $\bP(\tau_1> n)$ follow from \cite[Thm. 8.7.3]{cf:BGT}. It is not possible in general to deduce the asymptotics of $\bP(\tau_1=n)$, which need not even be regularly varying. However, in certain cases, one can recover at least some behavior of $\bP(\tau_1=n)$ from that of $\bP(n\in\tau)$ when the latter is regularly varying; we call such a result a \emph{reverse renewal theorem}. Specifically, if the renewal function
\[
  U_n:=\sum_{k=0}^n \bP(k\in\tau), \quad n\leq\infty,
\]
is slowly varying (as happens in the case of transient $\tau$ or $\alpha=0$), the following theorems apply.

\smallskip
\subsubsection*{Transient case}
We write $|\tau|$ for $|\{\tau_0,\tau_1,\dots\}|$, which is geometrically distributed in the transient case, with $E(|\tau|)=1/p_\infty^\tau$.

\begin{theorem}[Theorem 1.4 in \cite{cf:ABzero}]
\label{thm:transient}
If $\bP(n\in\tau)$ is regularly varying and $\tau$ is transient, then
\[\bP(\tau_1 = n) \simn \frac{1}{\bE(|\tau|)^2} \, \bP(n\in\tau)\, .\]
\end{theorem}

\smallskip
\subsubsection*{Recurrent case}
If $U_n$ is growing to infinity as a slowly varying function, then we have only a weaker reverse renewal theorem corresponding to \eqref{alpha0}.
\begin{theorem}[Theorem 1.3 in \cite{cf:ABzero}]
\label{thm:AB}
If $\bP(n\in\tau)$ is regularly varying, and if $U_n$ is slowly varying, then there exists some $\gep_n \stackrel{n\to\infty}{\to} 0$ such that
\[ \frac{1}{\gep_n n} \sum_{k=(1-\gep_n)n}^n \bP(\tau_1=k) \simn (U_n)^{-2} \bP(n\in\tau)\, .\]
\end{theorem}
One can therefore obtain the local asymptotics of $\bP(\tau_1=n)$ from this last theorem when one can show $\bP(\tau_1=n)$ is approximately constant over an interval of length $o(n)$, as done in Lemma \ref{lem:regvarying}.

\section{ Proof of Theorem \ref{thm:Prhon2}\lowercase{(i), (ii)} }
\label{sec:alpha*pos}
In case (i) we have $U_n^* \sim \frac{1}{\ga^*} \psi^*(n)n^{\alpha^*}$, and in case (ii) $U_n^* = \sum_{j=1}^n \frac{\psi^*(j) }{ j}$ which is slowly varying.
Hence by \cite[Thm.\ 8.7.3]{cf:BGT}, in case (i),
\[
  \bP(\rho_1>n) \simn \frac{1}{\Gamma(1-\alpha^*)\Gamma(1+\alpha^*)} \frac{1}{U_n^*} \simn 
    \frac{\sin(\pi\alpha^*)}{\pi} \psi^*(n)^{-1}\, n^{-\alpha^*},
\]
and in case (ii),
\begin{equation}\label{Unrn}
  \bP(\rho_1>n) \simn  \frac{1}{U_n^*} = \left(  \sum_{j=1}^n \frac{\psi^*(j)}{j} \right)^{-1}.
\end{equation}

\section{ Proof of Theorem \ref{thm:Prhon2}(iii) }
\label{sec:alpha*0inf}

For $\ga^*\geq 1$ (i.e.\ $\ga\geq 1$), we cannot extract  the behavior of $\bP(\rho_1>n)$ directly from that of $U_n^*$ as in Section \ref{sec:alpha*pos}, and we need a preliminary result: we prove that $\bP(\rho_1>n)$ is regularly varying and hence for any $\gep>0$ we have
\begin{equation}\label{gepeffect}
\bP(\rho_1>\gep n) = O(\bP(\rho_1>n)) \quad \text{as } n\to\infty.
\end{equation}

In Section \ref{sec:gepeffectinf}, we prove \eqref{gepeffect}, with the help of \cite{cf:Frenk}.
In Section \ref{sec:upperbound}, we prove an upper bound for $\bP(\rho_1>n)$.
Finally, in Section \ref{sec:lowerbound}, we prove the corresponding lower bound.

\subsection{Proof of \eqref{gepeffect}}
\label{sec:gepeffectinf}

A sequence $\{u_n\}$ is said to be in the \emph{de Haan class} $\Pi$ if there exists a slowly varying sequence $\ell_n$ such that for all $\lambda>0$,
\[
  \frac{u_{\lfloor\lambda n\rfloor}-u_n}{\ell_n} \to \log \lambda\quad \text{as } n\to\infty.
\]
We write $\mathrm{RVS}_{-\ga}$ for the set of regularly varying sequences  of index $-\ga$. We can state the results of Frenk \cite{cf:Frenk} as follows.
\begin{proposition}[\cite{cf:Frenk}, main theorem and Lemma 4]
\label{prop:Frenk} Let $\nu$ be a renewal process, and denote $u_n=\bP(n\in\nu)$. Then, we have
\begin{equation}\label{Frenk1}
\bP(\nu_1>n) \in \mathrm{RVS}_{-1} \quad \Leftrightarrow \quad u_n\in \Pi \,  .
\end{equation}
Moreover, for any $\ga>1$, denoting $m=\bE[\nu_1]<+\infty$, we have
\begin{equation}\label{Frenk2}
\bP(\nu_1>n) \in \mathrm{RVS}_{-\ga} \quad \Leftrightarrow \quad  u_n-\frac{1}{m} \in \mathrm{RVS}_{1-\ga}\, , 
\end{equation}
and each implies that
\begin{equation}\label{unerror}
  u_n-\frac{1}{m} \simn \frac{1}{m^2(\ga-1)}\, n \bP(\nu_1>n) \, .
  \end{equation}
\end{proposition}

Using Proposition \ref{prop:Frenk}, we prove that $\bP(\rho_1>n)$ is regularly varying with exponent $-\ga$, as follows, yielding \eqref{gepeffect}.

\medskip
If $\ga=\tga=1$, then Proposition \ref{prop:Frenk} tells that the slowly varying sequences $u_n=\bP(n\in\tau),\tilde{u}_n=\bP(n\in\sigma)$ are both in $\Pi$, with some corresponding slowly varying sequences $\ell_n,\tilde\ell_n$. (One expects $\ell_n\sim \gp(n)$ but we do not have or need proof of this.) Therefore, letting $L_n := \tilde\ell_nu_n + \ell_n\tilde{u}_n$, the product sequence $\bP(n\in\rho)=u_n\tilde{u}_n$ satisfies 
\begin{align}\label{deHaanprod}
 \frac{ u_{\lfloor\lambda n\rfloor}\tilde{u}_{\lfloor\lambda n\rfloor} - u_n\tilde{u}_n }{L_n} =  
   \frac{ u_{\lfloor\lambda n\rfloor} }{u_n} \frac{ \tilde{u}_{\lfloor\lambda n\rfloor} -\tilde{u}_n }{\tilde\ell_n} \frac{\tilde\ell_nu_n}{L_n}
   + \frac{ u_{\lfloor\lambda n\rfloor} - u_n }{\ell_n} \frac{ \ell_n\tilde{u}_n }{L_n} \stackrel{n\to\infty}{\to} \log \lambda
\end{align}
for all $\lambda>0$, so the product sequence is in $\Pi$. Applying Proposition \ref{prop:Frenk} again, we see that $\bP(\rho_1>n)$ is regularly varying with index $-1$.

\medskip

If $\ga=1$, $\tga>1$, then $\{u_n\}$ is in $\Pi$ (with some corresponding slowly varying sequence $\ell_n$), and $\tilde u_n - \frac{1}{\tilde\mu}$ is regularly varying with index $1-\tga$. Hence,
\begin{align*}
\frac{ u_{\lfloor\lambda n\rfloor}\tilde{u}_{\lfloor\lambda n\rfloor} - u_n\tilde{u}_n }{\mu^{-1}\ell_n} = \frac{\tilde u_{\lfloor\lambda_n \rfloor} }{\mu^{-1}} \frac{u_{\lfloor \gl n\rfloor } - u_n  }{\ell_n} + u_n \frac{\tilde u_{\lfloor \gl n\rfloor } -\tilde u_n }{\mu^{-1}\ell_n} \stackrel{n\to\infty}{\to} \log \gl ,
\end{align*}
where we used that $\tilde u_{\lfloor \gl n\rfloor } -\tilde u_n $ is in $\mathrm{RVS_{1-\tga}}$ so that the second term in the sum goes to $0$ (since $u_n /\ell_n$ is regularly varying with index $0$).
Hence $\bP(n\in\rho)=u_n\tilde{u}_n$ is in $\Pi$, and applying Proposition \ref{prop:Frenk}, we get that $\bP(\rho_1>n)$ is regularly varying with index $-1$.

\medskip

If $1<\ga\leq \tga$, then using Proposition \ref{prop:Frenk}, we get that
\begin{align}
\label{eq:Frenk}
u_n \tilde u_n - \frac{1}{\mu \tilde\mu} & = \Big(\frac{1}{\mu} + \frac{1+o(1)}{\mu^2(\ga-1)} n\bP(\tau_1>n) \Big) \Big(\frac{1}{\tilde\mu} +  \frac{1+o(1)}{\tilde\mu^2(\ga-1)} n\bP(\sigma_1>n) \Big) - \frac{1}{\mu\tilde\mu} \notag\\
&= \frac{1+o(1)}{\tilde\mu \mu^2 (\tga-1)} \, n \bP(\tau_1>n)  +  \frac{1+o(1)}{\mu \tilde\mu^2(\tga-1)} \,n\bP(\sigma_1>n), 
\end{align}
and therefore $u_n \tilde u_n - \frac{1}{\mu \tilde\mu} \in \mathrm{RVS}_{1-\ga}$. 
Applying Proposition \ref{prop:Frenk} again, we get that $\bP(\rho_1>n)$ is regularly varying with index $-\ga$ , and so \eqref{gepeffect} is proven.
Proposition~\ref{prop:Frenk} and \eqref{eq:Frenk} further give that
\begin{align}
\label{Prhon:Frenk}
\bP(\rho_1>n) & = (1+o(1)) \frac{1}{n}(\mu \tilde\mu)^2 (\ga-1) \Big( u_n \tilde u_n - \frac{1}{\mu\tilde\mu} \Big) \notag \\
& = (1+o(1)) \tilde \mu  \bP(\tau_1>n)  +   (1+o(1)) \mu \frac{\ga-1}{\tga-1} \bP(\sigma_1>n)
\end{align}
The second term is negligible compared to the first if $\tga>\ga>1$, so this proves Theorem~\ref{thm:Prhon2}(iii) when $1<\ga\leq \tga$.

We will present the rest of our proof of Theorem \ref{thm:Prhon2} in the whole range $1\leq \ga\leq \tga$ even though it is now needed only for $\ga=1$; this adds no complexity.  The advantage is that it is a more probabilistic approach, in that we use Proposition \ref{prop:Frenk} only to get the regular variation of $\bP(\rho_1>n)$, and avoid using the un-probabilistic \eqref{unerror} (with $\nu=\rho$) to estimate $\bP(\rho_1>n)$ as in \eqref{Prhon:Frenk}.
The method also provides an interpretation of the terms $\mu_n,\tilde\mu_n$ appearing in Theorem~\ref{thm:Prhon2}(iii).

\subsection{Some useful preliminary lemmas}

Before we prove Theorem \ref{thm:Prhon2}(iii), we need two technical lemmas. 

\begin{lemma}\label{lem:EK}
Let $\tau,\sigma$ be independent renewal processes, suppose $\rho=\tau\cap\sigma$ is recurrent with $\bE(\sigma_1)<\infty$, and let $K:=\min\{k\geq 1 : \tau_k \in\sigma\}$.  Then $\bE(K)=\bE(\sigma_1)$.
\end{lemma}

\begin{proof}
Since $\bP(n\in\rho)=\bP(n\in\tau)\bP(n\in\sigma)$, the renewal theorem gives 
\begin{equation}\label{Erho1}
  \bE(\rho_1)=\bE(\sigma_1)\bE(\tau_1).
  \end{equation}
Let $K_1,K_2,\dots$ be i.i.d.~copies of $K$ and let $S_m:=K_1+\dots+K_m$.  Then $\tau_{S_m}$ has the distribution of $\rho_m$, so using \eqref{Erho1},
\[
  \frac{\tau_{S_m}}{m}\to \bE[ \rho_1] = \bE[ \tau_1]\bE[\sigma_1] \ \ \text{ a.s.}, \quad \text{and}\quad
  \frac{\tau_{S_m}}{m}= \frac{\tau_{S_m}}{S_m} \frac{S_m}{m} \to \bE[ \tau_1] \bE[ K] \ \  \text{ a.s.},
  \]
and the lemma follows.
\end{proof}

Write $\bP_{x,y}(\cdot)$ for $\bP(\cdot\mid \tau_0=x,\sigma_0=y)$, and write $\bE_{x,y}$ the corresponding expectation.

\begin{lemma}\label{lem:N=0}
Assume \eqref{def:alphas}, and suppose $\rho$ is recurrent and $\ga^*>0$ (equivalently, $\ga+\tga>1$.)
Given $\eta>0$, provided $\gd$ is sufficiently small we have for large $n$ and all $0\leq x \leq \gd n$:
\begin{equation}\label{N=0}
  \bP_{-x,0}(\rho\cap [0,n]=\emptyset) < \eta \, .
\end{equation}
If also $\ga\geq 1$, then the same is true with $\gd>0$ arbitrary.  The analogous results with $\tau,\sigma$ interchanged hold as well.
\end{lemma}

\begin{proof}
Fix $x\leq \gd  n$ and let $N := |\rho\cap [0,n]|$. Then $\bP_{-x,0}(\rho\cap [0,n]=\emptyset) = \bP_{-x,0}(N=0)$
and $\bE_{-x,0}(N\mid N\geq 1) \leq U_n^*$, so
\begin{equation}\label{N0}
  \bP_{-x,0}(N=0) = \frac{\bE_{-x,0}(N\mid N\geq 1) - \bE_{-x,0}(N)}{\bE_{-x,0}(N\mid N\geq 1)} \leq \frac{U_n^* - \bE_{-x,0}(N)}{U_n^*}
\end{equation}
while
\begin{align}\label{differ}
  U_n^* - \bE_{-x,0}(N) &= \sum_{j=0}^n \bP(j\in\sigma) \big[ \bP(j\in\tau) - \bP(j+x\in\tau) \big]\, .
\end{align}
Since $\bP(j\in\tau)$ is regularly varying, given $\eta>0$, there exists $A$ (large) such that for $\gd>0$, for $n$ large we have for all $x\leq \gd n$ and $A\gd n\leq j\leq n$ that 
\begin{equation}\label{shift}
  \bP(j\in\tau) - \bP(j+x\in\tau) \leq \frac{\eta}{2}\bP(j\in\tau)\, .
  \end{equation}
Since $U_k^*$ is regularly varying, with positive index since $\ga^*>0$, if $\gd$, and therefore $A\gd$, is sufficiently small then for large $n$ we have $U_{A\gd n}^* \leq \frac{\eta}{2} U_n^*$.
With \eqref{differ} this gives that for large $n$,
\begin{equation}\label{twobounds}
  U_n^* - \bE_{-x,0}(N) \leq  U_{A\gd n}^* + \frac{\eta}{2} U_n^* \leq \eta U_n^*.
\end{equation}
With \eqref{N0}, this proves \eqref{N=0} for large $n$. 

\smallskip
Now consider $\ga\geq 1$, meaning $\bP(k\in\tau)$ is slowly varying.  Given $\eta>0$, for any $\gd>0$ we can choose $A$ (small this time) so that $U_{A\gd n}^* \leq \frac{\eta}{2} U_n^*$ for large $n$.  Inequality \eqref{shift} holds for all $j\geq A\gd n$ and $x\leq\gd n$, for $n$ large, so \eqref{twobounds} is valid and \eqref{N=0} follows.
\end{proof}

\subsection{Upper bound for $\bP(\rho_1>n)$}
\label{sec:upperbound}

Let us fix $\gep>0$.
Let us call a gap $\tau_k-\tau_{k-1}$ or $\sigma_k-\sigma_{k-1}$ \emph{long} if it exceeds $(1-2\gep)n$; the \emph{starting} and \emph{ending points} of such a gap are $\tau_{k-1},\tau_k$ or $\sigma_{k-1},\sigma_k$. Let $S$ be the first starting point of a long gap in $\tau$ or $\sigma$, and let $T$ be the ending point of the gap that starts at $S$. (To make things well-defined, if both $\tau$ and $\sigma$ have long gaps starting at $S$, then we take $T$ to be the first endpoint among these two gaps.)
Then
\begin{align}\label{2cases}
  \bP(\rho_1>n) &\leq \bP(\rho_1>n, \sigma \cap [\gep n, (1-\gep)n] \neq \emptyset, \tau \cap [\gep n, (1-\gep)n] \neq \emptyset)\notag\\
    &\hspace{3.5in} + \bP(\rho_1\geq T) .
\end{align}

\smallskip
For fixed $n$, we let $\bar\tau_1$ have the distribution of $\tau_1$ given $\tau_1 \leq (1-2\gep)n$, and similarly for $\bar\sigma_1$.  Let $\bar\tau$ and $\bar\sigma$ be renewal processes with gaps distributed as $\bar\tau_1$ and $\bar\sigma_1$, respectively, and let $K:=\min\{k\geq 1 : \tau_k \in\sigma\}$ and $\bar{K} := \min\{k\geq 1:\bar\tau_k \in \bar\sigma\}$. Then, we have 
\begin{align}\label{nolong}
  \bP&(\rho_1\geq T,S\in\tau) \notag\\
  &= \sum_{k\geq 0}
    \bP\bigg(K>k, \tau_i-\tau_{i-1} \leq (1-2\gep)n \text{ for all } i\leq k, \tau_{k+1}-\tau_k>(1-2\gep)n,\notag\\
  &\hskip 4cm \sigma_i-\sigma_{i-1} \leq (1-2\gep)n \text{ for all } i\text{ with } \sigma_{i-1}\leq \tau_k \bigg) \notag\\
  &\leq \sum_{k\geq 0} \bP(\bar{K}>k) \bP(\tau_1>(1-2\gep)n) \notag\\
  &= \bE[\bar{K}]\bP(\tau_1>(1-2\gep)n).
\end{align}
From Lemma \ref{lem:EK} we have $\bE[\bar{K}] = \bE(\sigma_1 \mid \sigma_1\leq (1-2\gep)n) \leq \tilde\mu_n$.  Thus for large $n$ we have 
\[
  \bP(\rho_1\geq T,S\in\tau) \leq (1-3\gep)^{-\ga}\tilde\mu_n \bP(\tau_1>n).
\]
A similar computation holds for $\bP(\rho_1\geq T,S\in\sigma)$ so we have for large $n$:
\begin{equation}\label{nolong2}
  \bP(\rho_1\geq T) \leq (1-3\gep)^{-\tga}\left\{ \tilde\mu_n \bP(\tau_1>n) + \mu_n\bP(\sigma_1>n) \right\}.
\end{equation}  

\smallskip
We now need a much smaller bound for the first term on the right side of \eqref{2cases}.  Define
$U := \min \tau\cap (\gep n,\infty)$ and $V := \min \sigma\cap (\gep n,\infty).$
Then
\begin{align}\label{uvsum}
  \bP\big(\rho_1&>n, \sigma \cap (\gep n, (1-\gep)n) \neq \emptyset, \tau \cap (\gep n, (1-\gep)n) \neq \emptyset,U<V \big) \notag\\
  &\leq \sum_{u<v,u,v\in(\gep n, (1-\gep)n)} \bP(\rho_1>\gep n,U=u,V=v) \bP_{u-v,0}(\rho_1>\gep n).
\end{align}

We may now apply Lemma \ref{lem:N=0} for the last probability. Fix $\eta>0$. Then, since $\tga\geq \ga\geq 1$ for $n$ large enough,
\begin{align}\label{unifclaim}
  \bP_{-x,0}(\rho_1>\gep n) < \eta \quad\text{for all } 0\leq x\leq n.
\end{align}
Therefore, summing  over $u,v$, the right side of \eqref{uvsum} is bounded by $\eta \bP(\rho_1>\gep n, U<V)$, and a similar bound holds when $U>V$. Hence, combining this with with \eqref{2cases} and \eqref{nolong2}, we get that
\begin{equation}
\bP(\rho_1>n) \leq (1-3\gep)^{-\tga} \left\{ \tilde\mu_n \bP(\tau_1>n) + \mu_n\bP(\sigma_1>n) \right\} + \eta \bP(\rho_1> \gep n)\, .
\end{equation}
Now we may use \eqref{gepeffect} to control the last term: we finally get that, provided $\eta$ is small enough, for large $n$,
\begin{equation}\label{finalupper}
  \bP(\rho_1>n) \leq (1+4\tga\gep) \left\{ \tilde\mu_n \bP(\tau_1>n) + \mu_n\bP(\sigma_1>n) \right\}\, .
\end{equation}

\subsection{Lower bound for $\bP(\rho_1>n)$}
\label{sec:lowerbound}
We use a modification of our earlier truncation. Fix $n$ and, analogously to $\bar\tau, \bar\sigma$, let $\hat\tau$ and $\hat\sigma$ be renewal processes with gaps $\hat\tau_i-\hat\tau_{i-1}=(\tau_i-\tau_{i-1})\wedge (n+1)$ and $\hat\sigma_i-\hat\sigma_{i-1}=(\sigma_i-\sigma_{i-1})\wedge (n+1)$, respectively, and let $\hat\rho=\hat\tau \cap \hat\sigma$ and $\hat K := \min\{k\geq 1:\hat\tau_k \in \hat\sigma\}$. We call a gap in $\hat\tau$ or $\hat\sigma$ \emph{large} if its length is $n+1$.  Let $[S_{\hat\tau},T_{\hat\tau}]$ and $[S_{\hat\sigma},T_{\hat\sigma}]$ be the first large gaps in $\hat\tau$ and $\hat\sigma$ respectively, and let $J_{\hat\tau}$ and $J_{\hat\sigma}$ be the number of large gaps in $\hat\tau$ and $\hat\sigma$ respectively before time $\hat\rho_1^{(n)}$.

Observe that
\begin{equation}\label{lower1}
  \bP(\rho_1> n) = \bP(\hat\rho_1 > n) \geq \bP(J_{\hat\tau}\geq 1) + \bP(J_{\hat\sigma}\geq 1) - \bP(J_{\hat\tau}\geq 1,J_{\hat\sigma}\geq 1).
\end{equation}
We claim that 
\begin{equation}\label{claimLtau}
  \bP(J_{\hat\tau}\geq 1) \geq (1-o(1))\bE[J_{\hat\tau}] \quad \text{as } n\to \infty
\end{equation}
and
\begin{equation}\label{claim2L}
  \bP(J_{ \hat\tau } \geq 1,J_{ \hat\sigma } \geq 1) = o\left( \bP(J_{\hat\tau}\geq 1)  +\bP(J_{\hat\sigma}\geq 1)  \right) \quad \text{as } n\to\infty.
\end{equation}

\medskip
Assuming \eqref{claimLtau} and \eqref{claim2L}, we have 
\begin{equation}\label{lower1a}
  \bP(\rho_1> n) \geq (1-o(1))\bigg( \bE[J_{\hat\tau}] + \bE[J_{\hat\sigma}] \bigg).
\end{equation}
Then using Lemma \ref{lem:EK} to get $\bE[\hat K]=\bE[\hat\sigma_1]=\tilde\mu_{n+1}$ we obtain
\begin{align}\label{ELlower}
  \bE[J_{\hat\tau}] & \sum_{k\geq 0} \bP\left(\tau_{k+1} - \tau_k > n , \hat K>k \right)\notag\\
  &= \bE[\hat K]\bP(\tau_1>n)= \tilde\mu_{n+1}\bP(\tau_1>n),
\end{align}
and similarly for $\bE[J_{\hat\sigma}]$.  With \eqref{lower1a} this shows that
\begin{equation}
  \bP(\rho_1>n) \geq (1-o(1)) \left\{ \tilde\mu_n \bP(\tau_1>n) + \mu_n\bP(\sigma_1>n) \right\}\, .
\end{equation}
This and \eqref{finalupper} prove Theorem \ref{thm:Prhon2}(iii).

\medskip
It remains to prove \eqref{claimLtau} and \eqref{claim2L}.  We begin with \eqref{claim2L}. 
We write
\begin{align}
\bP(J_{ \hat\tau } \geq 1,J_{ \hat\sigma } \geq 1) = \bP(J_{ \hat\tau } \geq 1&,J_{ \hat\sigma } \geq 1,S_{\hat\tau} <S_{\hat\sigma} )  \notag  \\
&  +\bP(J_{ \hat\tau } \geq 1,J_{ \hat\sigma } \geq 1,S_{\hat\tau} >S_{\hat\sigma} ) , \label{claim2L1}
\end{align}
and we control both terms separately. 
On the event $\{S_{\hat\tau} <S_{\hat\sigma} \}$, we decompose over the first $\hat\sigma$ renewal in the interval $(S_{\hat\tau}, T_{\hat\tau})$, to obtain that 
\begin{equation}
\label{claim2L2}
\bP(J_{ \hat\tau } \geq 1,J_{ \hat\sigma } \geq 1,S_{\hat\tau} <S_{\hat\sigma} )  \leq \bP(J_{ \hat\tau } \geq 1) \times \sup_{x\in(0,n]} \bP_{x,0} \left ( J_{ \hat\sigma } \geq 1 \right ) .
\end{equation}
From Lemma \ref{lem:N=0} we have that for any $\eta>0$, for $n$ large enough, for all $1\leq x\leq  n/2$,
\begin{equation}
\label{claim2L3}
\bP_{x,0} \left ( J_{ \hat\sigma } \geq 1 \right ) \leq \bP_{x,0} \left ( \rho_1 \geq n/2 \right ) \leq \eta \, .
\end{equation}
If $x \in (n/2,n]$, then we decompose over the first $\sigma$ renewal in the interval $[x/2,x)$ if it exists, to get
\begin{equation}
\bP_{x,0} \left ( J_{ \hat\sigma } \geq 1 \right ) \leq \bP( \sigma\cap[x/2,x) =\emptyset ) + \sup_{y\in [1,x/2] }\bP_{y,0} \left ( J_{ \hat\sigma } \geq 1 \right ).
\end{equation}
The last sup in bounded as in \eqref{claim2L3}. For the first probability on the right, using the renewal theorem when $\ga>1$ and \cite{cf:Eric} when $\ga=1$, we get that there is a constant $c_{5}$ such that
\[\bP( \sigma\cap[x/2,x) =\emptyset) \leq \sum_{k=1}^{x/2} \bP(k\in\sigma) \bP(\sigma_1 >x/2) \leq c_{5} \frac{x}{\mu_x} \gp(x) x^{-\ga} \to 0\ \text{as } x\to\infty.\]
The convergence to $0$ is straightforward when $\ga>1$, and uses that $\gp(x) /\mu_x \to 0$ as $x\to\infty$ when $\ga=1$ (see for example Theorem 1 in \cite[Ch. VIII, Sec. 9]{cf:Feller2}).
It follows that the sup in \eqref{claim2L2} approaches 0 as $n\to\infty$. The second probability on the right side of 
\eqref{claim2L1} is handled similarly, and this proves \eqref{claim2L}.

\medskip
We now turn to \eqref{claimLtau}. We  show that for any $\eta>0$, we can take $n$ large enough so that for any $j\geq 1$, 
\begin{equation}\label{claimLtau1}
\bP(J_{\hat\tau} \geq j +1) \leq \eta \bP\left(  J_{\hat\tau} \geq j  \right).
\end{equation}
This easily gives that $\bE \left[ J_{\hat\tau} \right] = \sum_{j\geq 1} \bP(J_{\hat\tau} \geq j) \leq \frac{1}{1-\eta} \bP(J_{\hat\tau} \geq 1 )$, which is \eqref{claimLtau}. To prove \eqref{claimLtau1}, we denote $T^{(j)}_{\hat\tau}$ the endpoint of the $j{\rm th}$ large gap in $\hat\tau$. Then, decomposing over the first $\hat \sigma$ renewal in the interval $[T^{(j)}_{\hat\tau}-n, T^{(j)}_{\hat\tau})$, we get, similarly to \eqref{claim2L2}
\begin{align*}
\bP(J_{\hat\tau} \geq j +1) & \leq \bP\left(  J_{\hat\tau} \geq j  \right) \times \sup_{x\in(0,n]} \bP_{0,-x} \left(  J_{\hat\tau} \geq 1 \right) \\
& \leq  \bP\left(  J_{\hat\tau} \geq j  \right) \sup_{x\in(0,n]} \bP_{0,-x} \left( \rho_1\geq n+1 \right) \leq \eta \,  \bP\left(  J_{\hat\tau} \geq j  \right) ,
\end{align*}
where the last inequality is valid provided that $n$ is large enough, thanks to Lemma~\ref{lem:N=0}.  This completes the proof of \eqref{claimLtau}, and thus also of Theorem \ref{Prhon2}(iii).

\section{Proof of Lemma \ref{lem:regvarying}: Stretching of gaps}
\label{sec:regvary}

By assumption $\rho$ is recurrent, and we need to show that when $n$ is large $\bP(\rho_1=n) \approx \bP(\rho_1 = n+k)$ for all $k\in(0,\gep n)$, with $\gep\ll 1$.
The idea is to take the set of trajectories of $\tau$ and $\sigma$ such that $\rho_1=n$, and to \emph{stretch} them slightly so that $\rho_1=n+k$, see Figure~\ref{fig:stretch}.
In Section \ref{sec:largegap}, we prove that for some $\gd>0$, conditioned on $\rho_1=n$, the largest gap of $\tau$ and $\sigma$ in $[0,n]$ is larger than $\gd n$ with high probability; see Lemma \ref{lem:longjump}. Assume that it is a $\tau$-gap, and that it has length $m$.
Then, in Section \ref{sec:stretching}, we show that for $\gep\ll\delta$ we can stretch this $\tau$-gap by $k \leq\gep n \ll m$, and stretch $\sigma$ inside this $\tau$-gap by the same $k$, without altering the probability significantly.

\begin{figure}[htbp]
\begin{center}
\setlength{\unitlength}{2mm}
\begin{picture}(58,14)(0,0)
\put(0,0) {\includegraphics[width=120mm]{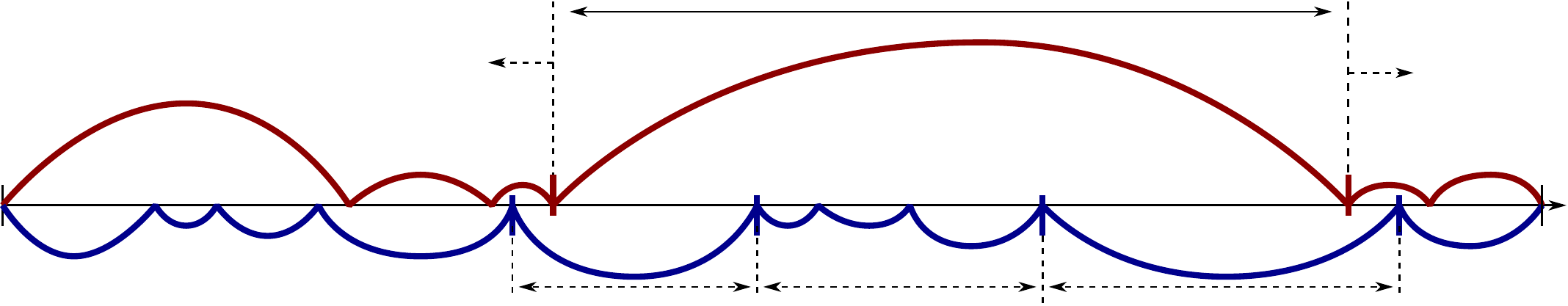}}
\put(5,0) {$\sigma$}
\put(5,8) {$\tau$}
\put(-0.5,5) {$0$}
\put(58,5.5) {$\rho_1=n$}
\put(30,12) {$\tau_i -\tau_{i-1}=m\geq \gd n$}
\put(18.4,1) {$p$}
\put(23,-1) { \small $t_1$}
\put(34,-1) {\small $t_2$}
\put(46.4,-1) {\small $t_3$}
\put(21.5,2.2) {$j$}
\end{picture}
\end{center}
\caption{
How to ``stretch'' trajectories, to go from $\rho_1=n$ to $\rho_1=n+k$ :  we identify the largest gap in $\tau$ (which is larger than $\gd n$ with great probability, see Lemma \ref{lem:longjump}) and we stretch it by $k$, while at the same time stretching one of the three associated $\sigma$-intervals (the largest of $t_1,t_2,t_3$). See the proof of Lemma \ref{lem:stretch} for more detailed explanations.}
\label{fig:stretch}
\end{figure}
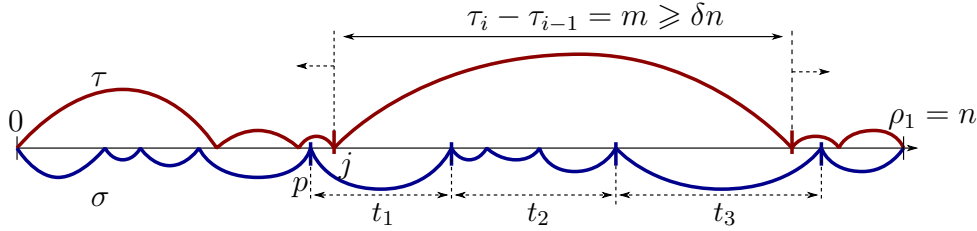

\subsection{Probability of having a large gap}
\label{sec:largegap}

Denote by $\mA_{\gd}$ the event that there is a gap (either in $\sigma$ or $\tau$) longer than $\gd n$:
\begin{equation}
\label{def:longjump}
\mA_\gd:= \big\{ \exists \, i:\,  \tau_i-\tau_{i-1} > \gd n  \, , \, \tau_i\leq n \, \text{ or } \, \sigma_{i}- \sigma_{i-1} > \gd n\, ,\,  \sigma_i \leq n  \big\} \, .
\end{equation}
We will show that $\mA_{\gd}^c$ contributes only a small part of $\{\rho_1=n\}$. Recall that 
\[v_n = \bP(\rho_1 > n)^2 \bP(n\in\rho)\, . \]

\begin{lemma}
\label{lem:longjump}
Assume \eqref{def:alphas}. There exist $c_{6}>0$ and $\gd_0$ such that if $\gd\in(0,\gd_0)$, then for $n$ sufficiently large, 
\[ \bP\Big(\rho_1=n\, ;\,  \mA_{\gd}^c \Big) \leq e^{-c_{6}/\gd} v_n \, .\]
\end{lemma}

\begin{proof}
On the event $\{\rho_1=n\}\cap\mA_{\gd}^c$, all $\tau$ and $\sigma$ gaps are smaller than $\gd n$, and therefore all blocks of length at least $\gd n$ are visited by both $\tau$ and $\sigma$. We control probabilities in each third of $[0,n]$ separately.
To that end, define
\[
  \ell_{\tau}= \max \tau\cap(0,n/3), \quad \ell_{\sigma} =\max \sigma\cap(0,n/3),
\]
and define events
\begin{equation}\label{GG1G3}
  G_1: \tau\cap\sigma \cap (0,n/8) = \emptyset,\qquad G_2: \tau\cap\sigma \cap [n/3,2n/3] = \emptyset,
\end{equation}
\[
  G_3: \tau\cap\sigma \cap (7n/8,n) = \emptyset,
\]
\[
  \mD_{\gd\tau}: \tau_i-\tau_{i-1} \leq \gd n \text{ for all $i$ with } [\tau_{i-1},\tau_i] \cap [n/3,2n/3] \neq \emptyset,
\]
\[
  \mD_{\gd\sigma}: \sigma_i-\sigma_{i-1} \leq \gd n \text{ for all $i$ with } [\sigma_{i-1},\sigma_i] \cap [n/3,2n/3] \neq \emptyset,
\]
\[
  \mL_1: \ell_\tau,\ell_\sigma \in (n/4,n/3).
\]
Assuming $\gd<1/12$, we have $\mA_\gd^c \subset \mD_{\gd\tau}\cap\mD_{\gd\sigma} \subset \mL_1$.

\medskip
{\bf End thirds.}
By Lemma \ref{lem:appendix}, there exists $C_0$ such that
\begin{align}\label{revcond}
\max_{i,j\in (n/4,n/3)} \bP\big( G_1 \mid \ell_\tau=i,\ell_\sigma=j ) &= \max_{i,j\in (n/4,n/3)} \bP\big( G_1 \mid i\in\tau,j\in\sigma ) \leq C_0 \bP(G_1).
\end{align}
It follows that
\begin{align}\label{elleffect}
  \bP(\rho_1=&n,\mA_\gd^c \mid n\in\rho) \leq \bP\left( G_1\cap G_2\cap G_3 \cap \mD_{\gd\tau}\cap\mD_{\gd\sigma}\ \big|\ n\in\rho\right)
    \notag\\
  &= \bP\left( G_1\mid G_2\cap G_3 \cap \mD_{\gd\tau}\cap\mD_{\gd\sigma}\cap\{n\in\rho\} \right)  \, \bP\left( G_2\cap G_3 \cap \mD_{\gd\tau}\cap\mD_{\gd\sigma} \mid n\in\rho\right) \notag\\
  &= \bE\left( \bP(G_1 \mid \ell_\tau,\ell_\sigma)\ \big|\ G_2\cap G_3 \cap \mD_{\gd\tau}\cap\mD_{\gd\sigma}\cap\{n\in\rho\} \right) \notag\\
  &\hspace{5cm} \times \bP\left( G_2\cap G_3 \cap \mD_{\gd\tau}\cap\mD_{\gd\sigma}\mid n\in\rho\right) \notag\\
  & \leq C_0 \bP(G_1) \bP\left( G_2\cap G_3 \cap \mD_{\gd\tau}\cap\mD_{\gd\sigma}\mid  n\in\rho\right).
\end{align}
Symmetrically we obtain
\begin{align}\label{seg5}
  \bP\left( G_2\cap G_3 \cap \mD_{\gd\tau}\cap\mD_{\gd\sigma}\ \big|\ n\in\rho\right) \leq C_0 \bP(G_3)
    \bP\left( G_2\cap \mD_{\gd\tau}\cap\mD_{\gd\sigma}\ \big|\ n\in\rho \right)
\end{align}
so, using Theorem \ref{thm:Prhon2},
\begin{align}\label{seg15}
  \bP(\rho_1=n,\mA_\gd^c \mid n\in\rho) &\leq C_0 ^2\bP(\rho_1 \geq n/8)^2 
    \bP\left( G_2\cap \mD_{\gd\tau}\cap\mD_{\gd\sigma}\ \big|\ n\in\rho \right) \notag\\
  &\leq c_{7} \bP(\rho_1>n)^2 \bP\left( G_2\cap \mD_{\gd\tau}\cap\mD_{\gd\sigma}\ \big|\ n\in\rho \right) .
\end{align}

\medskip
{\bf Middle third.}
We need to bound the last probability in \eqref{seg15}.
We divide the interval $[n/3,2n/3]$ into blocks $B_i=[a_{i-1},a_i]$ of length $A \gd n$ where $A$ is a (large) constant to be specified.
We denote by $d_\tau^{(i)}$ and $f_\tau^{(i)}$ the first and last renewals, respectively, of $\tau$ in  $B_i$, and similarly for $d_\sigma^{(i)},f_\sigma^{(i)}$. Let $B_{i,\ell}:= [a_{i-1},a_{i-1}+\gd n]$ and $B_{i,r}:=[a_i-\gd n,a_i]$.
On the event $\mD_{\gd\tau}\cap\mD_{\gd\sigma}$, we  have $d_\tau^{(i)},d_\sigma^{(i)}\in B_{i,\ell}$ and $f_\tau^{(i)},f_\sigma^{(i)}\in B_{i,r}$.  Let $B_i^{(1)} := [a_{i-1},a_{i-1}+A\gd n/3]$ denote the first third of $B_i$.  Define events
\[
  \mD_{\gd\tau}^{(i)}: \tau_j - \tau_{j-1}\leq \gd n \text{ for all $j$ with } \tau_{j-1}\in B_i^{(1)},
\]
\[
  \mD_{\gd\sigma}^{(i)}: \sigma_j - \sigma_{j-1}\leq \gd n \text{ for all $j$ with } \sigma_{j-1}\in B_i^{(1)}.
\]
Using again Lemma \ref{lem:appendix}, we obtain
\begin{align}\label{GDD}
  \bP&\left( G_2\cap \mD_{\gd\tau}\cap\mD_{\gd\sigma}\ \big|\ n\in\rho \right)\notag\\
  &\leq \prod_{i\leq 1/3 A\gd} \max_{h,k\in B_{i,\ell},\, j,m\in B_{i,r}}
    \bP \bigg( \tau\cap\sigma \cap B_i^{(1)} = \emptyset, \mD_{\gd\tau}^{(i)}, \mD_{\gd\sigma}^{(i)}\, \notag\\
  &\hskip 6cm \Big| \, d_\tau^{(i)}=h, f_\tau^{(i)}=j, d_\sigma^{(i)}=k, f_\sigma^{(i)}=m \bigg) \notag\\
  &\leq \prod_{i\leq 1/3 A\gd} \max_{h,k\in B_{i,\ell}} C_0\,
    \bP \big( \tau\cap\sigma \cap B_i^{(1)} = \emptyset, \mD_{\gd\tau}^{(i)}, \mD_{\gd\sigma}^{(i)}\,\big| \,    
    d_\tau^{(i)}=h, d_\sigma^{(i)}=k \big)\, .
\end{align}

We claim that for any $\eta>0$, there exists $A>0$ such that, for $\gd$ small, for $n$ large enough, for all $h,k\in[0,\gd n)$,
\begin{equation}
\label{proba<eta}
\bP_{h,k}\big( \tau\cap\sigma \cap (0, \tfrac13 A \gd n] =\emptyset\, ,\, \mD_{\gd\tau}^{(1)}, \mD_{\gd\sigma}^{(1)} \big) \leq \eta \, .
\end{equation}
This bounds all the probabilities on the right side of \eqref{GDD} by $\eta$, which with \eqref{seg15} and \eqref{GDD} shows that, provided $\eta$ is small,
\[
  \bP(\rho_1=n,\mA_\gd^c \mid n\in\rho) \leq c_{7}\bP(\rho_1>n)^2 (C_0 \eta)^{1/3A\gd} \leq e^{-c_{6}/\gd}\bP(\rho_1>n)^2,
\]
which completes the proof of the lemma.

\medskip
It remains to prove \eqref{proba<eta}. In the case of $\ga\geq 1,\tga\geq 1$, we can drop the events $\mD_{\gd\tau}^{(1)}, \mD_{\gd\sigma}^{(1)}$ and \eqref{proba<eta} follows from Lemma \ref{lem:N=0}. So suppose $\ga<1$; we will show that $\bP_{0,0}( \mD_{\gd\tau}^{(1)} ) \leq \eta$.  (This is sufficient, since $\bP_{h,k}( \mD_{\gd\tau}^{(1)} ) \leq \bP_{\gd n,0}( \mD_{\gd\tau}^{(1)} )$ for all $h,k\in [0,\gd n)$ and the last probability is unchanged if we replace $\gd n$ with $0$ and $\frac13 A$ with $\frac13A-1$.) We therefore drop the subscript $0,0$ in the notation. 

Let $J:= \min\{j\geq 1: \tau_j-\tau_{j-1}>\gd n\}$, let $\bar \tau_1$ have the distribution of $\tau_1$ given $\tau_1\leq \gd n$, and let  $\bar\tau$ be a renewal process with gaps distributed as $\bar \tau_1$. We have for $k\geq 1$:
\begin{align}\label{jsum}
  \bP( \mD_{\gd\tau}^{(1)} ) &\leq \sum_{j=0}^{k-1} \bP(J=j+1,\tau_j>A\gd n/3) + \bP(J> k)\notag\\
  &\leq \sum_{j=0}^{k-1} \bP(J=j+1)\bP\left( \bar \tau_j > \tfrac13 A\gd n \right) 
    +\bP \left( \max_{i\leq k} (\tau_i-\tau_{i-1}) \leq \gd n\right) , \notag \\
    & \leq \bP\left( \bar \tau_k > \tfrac13 A\gd n \right) + e^{-k \bP(\tau_1 >\gd n)}\ .
\end{align}
Then we use that for any $\ga \in [0,1)$ there exist some $c_{8},c_9>0$ such that for large $n$, $\bE[\bar \tau_1] \leq c_{8} \gp(n) (\gd n)^{1-\ga}$, and $\bP(\tau_1 >\gd n) \geq c_{9} \gp(n) (\gd n)^{-\ga}$ (in fact $\bP(\tau_1> \gd n) \gg \gp(n)$ for $\ga =0$.) We obtain that
\begin{equation}
  \bP( \mD_{\gd\tau}^{(1)} ) \leq \frac{3 c_{8}}{A} k \gp(n) (\gd n)^{-\ga} + e^{- c_{9} k \gp(n) (\gd n)^{-\ga}} .
\end{equation}
Choosing $k= A^{1/2} \gp(n)^{-1} (\gd n)^{\ga}$ with $A$ large enough, we get that $\bP( \mD_{\gd\tau}^{(1)} ) \leq \eta$. This completes the proof of \eqref{proba<eta}.
\end{proof}

\subsection{Stretching argument}
\label{sec:stretching}

We next show that, on the event $\mA_{\gd}$, we can formalize the stretching previously described, and the cost of the stretching is small.  

\begin{lemma}
\label{lem:stretch}
Assume \eqref{def:alphas}. Given $\gd>0$, if $n$ is sufficiently large, then for any $k\in [0,2\gd^3 n]$ we have
\[\bP(\rho_1=n \, ;\, \mA_{\gd}(n)) \leq (1+\gd) \bP(\rho_1=n+k)\, .\]
\end{lemma}

\begin{proof}
Fix $n$ and denote
\[
  M_{\tau}:= \max \{ \tau_i -\tau_{i-1}: \tau_i \leq n\} \quad \text{and} \quad M_{\sigma}:= \max \{ \sigma_i -\sigma_{i-1}:\, \tau_i \leq n\} \, ,
  \]
\[
  \mA_{\gd}^{\tau}(n) := \{\rho_1=n\} \cap \mA_\gd \cap \{M_{\tau}\geq M_{\sigma}\}, \quad 
    \mA_{\gd}^{\sigma}(n) := \{\rho_1=n\} \cap \mA_\gd \cap \{M_{\sigma}> M_{\tau}\}.
  \]
We will show that provided that $\gd$ is small enough, for $n$ large enough and $k\in[0,2\gd^3 n]$
\begin{equation}
\label{taulargerup}
\bP(\mA_{\gd}^{\tau}(n)) 
\leq  (1+\gd) \bP(\rho_1=n+k \, ,\, M_{\tau} \geq M_{\sigma})\, .
\end{equation}
The analogous statement also holds with $\mA_{\gd}^{\sigma}(n)$ instead of $\mA_{\gd}^{\tau}(n)$; combining the two completes the proof.

\medskip

To prove \eqref{taulargerup}, define random indices 
\[
  i_0:= \min\{ i\geq 1: \tau_i-\tau_{i-1}=M_{\tau}\}, \quad \ell_0:= \min\{\ell\geq 1:\sigma_\ell>\tau_{i_0-1}\},
  \]
\[
  \ell_1 := \min\{\ell\geq 1:\sigma_\ell\geq\tau_{i_0}\}.
\]
We call $[\tau_{i_0-1},\tau_{i_0}]$ the \emph{maximal gap}, and the three intervals $[\sigma_{\ell_i-1},\sigma_{\ell_i}], i=0,1$ and $[\sigma_{\ell_0},\sigma_{\ell_1-1}]$ are called \emph{associated $\sigma$-intervals}.
We decompose the probability according to the locations of this gap and the intervals:  define the events
\begin{align*}
  \mA_{\gd}^{\tau}(n,j,m,p,t_1,t_2,t_3) &:=\mA_{\gd}^{\tau}(n) \cap 
    \big\{ \tau_{i_0}=j,\tau_{i_0}-\tau_{i_0-1}=m,\sigma_{\ell_0-1}=p,\\
  &\qquad\qquad \sigma_{\ell_0} - \sigma_{\ell_0-1} =t_1,\sigma_{\ell_1-1}-\sigma_{\ell_0}=t_2,
    \sigma_{\ell_1}-\sigma_{\ell_1-1}=t_3 \big\}\, .
\end{align*}
This means the maximal gap (in $\tau$) is from $j$ to $j+m$, and $\sigma$ has gaps from $p$ to $p+t_1$ and from $p+t_1+t_2$ to $p+t_1+t_2+t_3$, each containing an endpoint of the maximal $\tau$ gap, see Figure \ref{fig:stretch}.
For the event to be nonempty, we must have $m\geq \gd n$ and
\begin{equation}\label{nonempty}
  0 \leq p<j<p+t_1\leq p+t_1+t_2< j+m\leq p+t_1+t_2+t_3\leq n.
\end{equation}
Given such indices let us define $I\leq 3$ by $t_I =\max\{t_1,t_2,t_3\}$, with ties broken arbitrarily.
Consider now the map $\Phi_k$ which assigns to each nonempty event $\mA_{\gd}^{\tau}(n,j,m,p,t_1,t_2,t_3)$ the event 
\[
  \Phi_k(\mA_{\gd}^{\tau}(n,j,m,p,t_1,t_2,t_3)) := \begin{cases} \mA_{\gd}^{\tau}(n+k,j,m+k,p,t_1+k,t_2,t_3) &\text{if } I=1,\\
    \mA_{\gd}^{\tau}(n+k,j,m+k,p,t_1,t_2+k,t_3) &\text{if } I=2,\\
    \mA_{\gd}^{\tau}(n+k,j,m+k,p,t_1,t_2,t_3+k) &\text{if } I=3. \end{cases}
  \] 
Applying $\Phi_k$ corresponds to stretching the maximal gap and the longest of the associated $\sigma$-intervals by the amount $k$.  It is easy to see that for distinct tuples $(j,m,p,t_1,t_2,t_3)$, the corresponding events $\Phi_k(\mA_{\gd}^{\tau}(n,j,m,p,t_1,t_2,t_3))$ are disjoint subsets of $\mA_{\gd}^{\tau}(n+k)$; this just means that the relevant interval and gap lengths in the original configuration are identifiable from the stretched configuration. We claim that provided $\gd$ is small enough, for $n$ large enough and $k\in[0,2\gd^3 n]$,  
\begin{equation}\label{deltaclaim}
  \bP\left( \mA_{\gd}^{\tau}(n,j,m,p,t_1,t_2,t_3) \right) \leq (1+\gd)
     \bP\left( \Phi_k\big(\mA_{\gd}^{\tau}(n,j,m,p,t_1,t_2,t_3) \big) \right)
\end{equation}
whenever $\mA_{\gd}^{\tau}(n,j,m,p,t_1,t_2,t_3) \neq\emptyset$. Due to the aforementioned disjointness, summing this over $(j,m,p,t_1,t_2,t_3)$ immediately yields \eqref{taulargerup}. To prove \eqref{deltaclaim}, note that if $I=1$ then $t_1\geq m/3$, so $k/t_1\leq 6\gd^2$, while $k/m<2\gd^2$, so provided $\gd$ is small,
\[
  \frac{ \bP\left( \mA_{\gd}^{\tau}(n,j,m,p,t_1,t_2,t_3) \right) }{ \bP\left( \Phi_k(\mA_{\gd}^{\tau}(n,j,m,p,t_1,t_2,t_3)) \right) }
    = \frac{\bP(\tau_1=m)}{\bP(\tau_1=m+k)} \frac{ \bP(\sigma_1=t_1) }{ \bP(\sigma_1=t_1+k) } < 1+\gd.
\]
The same bound holds if $I=3$. If $I=2$ we have $t_2\geq m/3$, so $k/t_2\leq 6\gd^2$, and provided that $\gd$ is small
\[
  \frac{ \bP\left( \mA_{\gd}^{\tau}(n,j,m,p,t_1,t_2,t_3) \right) }{ \bP\left( \Phi_k(\mA_{\gd}^{\tau}(n,j,m,p,t_1,t_2,t_3)) \right) }
    = \frac{\bP(\tau_1=m)}{\bP(\tau_1=m+k)} \frac{ \bP(t_2\in\sigma) }{ \bP(t_2+k\in\sigma) } < 1+\gd.
\]
The claim \eqref{deltaclaim}, and hence the lemma, now follow.
\end{proof}

We proceed with the proof of Lemma \ref{lem:regvarying}.
Indeed, the second inequality in \eqref{stretched} is immediate from Lemmas \ref{lem:longjump} and \ref{lem:stretch}.
Also, since $v_{n}$ is regularly varying, Lemma~\ref{lem:longjump} gives that for $\gd$ small, for any $j\in(0,\gd^3 n]$, 
\[
  \bP(\rho_1=n -j \, ;\, \mA_{\gd}^c(n-j)) \leq 2e^{-c_6/\gd} v_n \, .
  \]
This and Lemma \ref{lem:stretch} yield that for any $k\in(0,\gd^3 n] \subset (0,2\gd^3(n-k)]$,
\begin{equation}
\label{lowerbound}
\bP(\rho_1=n-k) \leq (1+\gd) \bP(\rho_1=n) + 2e^{-c_6/\gd} v_n \, .
\end{equation}
and the first inequality in \eqref{stretched} follows.

\section{Proof of Theorem \ref{thm:main2}}
\label{sec:regvarying}
Let
\[
  A_n^+(\gep) := \frac{\bP(\rho_1> n) - \bP(\rho_1> (1+\gep) n)}{\gep n }\, ,
\]
\[
  A_n^-(\gep) := \frac{\bP(\rho_1> (1-\gep)n) - \bP(\rho_1 > n)}{\gep n }\, .
\]
We claim that, if $\rho$ is recurrent,  there is a constant $c_{10}>0$ such that for sufficiently small $\gep>0$, when $n$ is large,
\begin{equation}
\label{eq:vn}
v_n \leq c_{10}A_n^\pm(\gep) \, .
\end{equation}
It is sufficient to prove this for $A_n^+(\gep)$, since $v_n$ is regularly varying.
Consider first $\ga^*=0$.  
It follows readily from \eqref{Unrn} and Theorem \ref{thm:AB} that for small $\gep$, when $n$ is large we have 
\begin{equation}\label{Angep}
  A_n^+(\gep) \geq \frac12 (U_n^*)^{-2}\bP(n\in\rho) \geq \frac14 v_n\, .
\end{equation}
Next consider $\ga^*\in (0,1)$.  Here $\ga^*=1-\theta^*$, so by Theorem \ref{thm:Prhon2}, for some $c_{11}$, for small $\gep$  we have for large $n$
\[
  A_n^+(\gep) \geq c_{11}n^{-(1+\ga^*)}\psi^*(n)^{-1} = c_{11}n^{-\theta^*}\psi^*(n) n^{-2\ga^*}\psi^*(n)^{-2} \geq c_{12}v_n\, .
\]
Finally consider $\ga^*\geq1$; here $1\leq\ga\leq\tga$. Since $\bP(\tau_1=n)$ and $\bP(\sigma_1=n)$ are regularly varying, it follows from Theorem \ref{thm:Prhon2} that for $\gep$ small and large $n$,
\[
  A_n^+(\gep) \geq \frac12 \left( \tilde\mu_n\bP(\tau_1=n) + \mu_n\bP(\sigma_1=n) \right) = \tilde\mu_n\frac{\gp(n)}{2n^{1+\ga}} 
    + \mu_n\frac{\tgp(n)}{2n^{1+\tga}}\, .
\]
Using $(a+b)^2\leq 2a^2+2b^2$, we obtain
\begin{align}
   v_n &\leq 2\left( \tilde\mu_n\frac{ \gp(n) }{n^{\ga}} + \mu_n\frac{ \tgp(n) }{n^{\tga}} \right)^2 \frac{1}{ \mu_n\tilde\mu_n }   
    \leq 4\left( \tilde\mu_n\frac{\gp(n)}{n^{2\ga}}\frac{\gp(n)}{\mu_n} + \mu_n \frac{ \tgp(n) }{n^{2\tga}} \frac{ \tgp(n) }{\tilde\mu_n} \right)   \leq A_n^+(\gep),
\end{align}
where for last inequality  we used that $\frac{\gp(n)}{n^{\ga-1}\mu_n} \to 0$ as $n\to\infty$ (since $\gp(n)/\mu_n\to 0$ when $\ga=1$), and similarly  $\frac{\tilde\gp(n)}{n^{\tga-1}\tilde\mu_n} \to 0$.
The claim \eqref{eq:vn} is now proved.

\medskip
For $\gd$ sufficiently small, applying Lemma \ref{lem:regvarying} and \eqref{eq:vn} we get that for $n$ large and $c_{13}=c_{10}+1$,
\begin{align}\label{eq:lowerbdPrhon}
\bP(\rho_1 = n) \geq (1-c_{13}\gd)A_n^-(\gd^3)\, .
\end{align}
Similarly, we get 
\begin{equation}
\label{eq:upperboundPrhon}
\bP(\rho_1 = n) \leq (1+c_{13} \gd) A_n^+(\gd^3) \, .
\end{equation}

If $\ga^*=0$, as with \eqref{Angep} it follows easily from Theorem \ref{thm:AB} that for large $n$ we have
\[
  A_n^-(\gd^3) \geq (1-\gd) (U_n^*)^{-2} \bP(n\in\rho) \quad \text{and} \quad 
    A_n^+(\gd^3) \leq (1+\gd) (U_n^*)^{-2} \bP(n\in\rho),
\]
and then part (ii) of the theorem follows from \eqref{eq:lowerbdPrhon} and \eqref{eq:upperboundPrhon}.

If $\ga^*\in (0,1)$, then by Theorem \ref{thm:Prhon2}(i), when $\gd$ is small we have for large $n$
\[
  A_n^-(\gd^3) \geq (1-\gd)\frac{ \alpha^*\, \sin(\pi \alpha^*) }{\pi }\,   \psi^*(n)^{-1} n^{-(1+\alpha^*)},
\]
\[
  A_n^+(\gd^3) \leq (1+\gd)\frac{ \alpha^*\, \sin(\pi \alpha^*) }{\pi }\,   \psi^*(n)^{-1} n^{-(1+\alpha^*)},
\]
and again part (i) of the theorem follows from \eqref{eq:lowerbdPrhon} and \eqref{eq:upperboundPrhon}.

If $\ga^*\geq 1$, then by Theorem \ref{thm:Prhon2}(iii), when $\gd$ is small we have for large $n$
\[
  A_n^-(\gd^3) \geq (1-\gd)\left( \tilde\mu_n\frac{\gp(n)}{n^{1+\ga}} + \mu_n\frac{\tgp(n)}{n^{1+\tga}} \right),
\]
\[
  A_n^+(\gd^3) \leq (1+\gd)\left( \tilde\mu_n\frac{\gp(n)}{n^{1+\ga}} + \mu_n\frac{\tgp(n)}{n^{1+\tga}} \right),
\]
and part (iii) of the theorem follows once more from \eqref{eq:lowerbdPrhon} and \eqref{eq:upperboundPrhon}.
\qed

\begin{appendix}
\section{Extension of Lemma A.2 in \cite{cf:GLT10}}

We generalize here Lemma A.2 of \cite{cf:GLT10}, which covers $\ga>0$, to include $\ga=0$.  The idea is essentially unchanged, but the computations are different.
\begin{lemma}
\label{lem:appendix}
Assume that $\bP(\tau_1=k) = \gp(k)k^{-(1+\ga)}$ for some $\ga\geq 0$ and some slowly varying function $\gp(\cdot)$.
Then, there exists a constant $C_0>0$ such that, for all sufficiently large $n$, for any non-negative function $f_n(\tau)$ depending only on $\tau\cap\{0,\ldots,n\}$, we have
\[\bE[f_n(\tau) \mid 2n \in\tau] \leq C_0 \bE[f_n(\tau)]\, .\]
\end{lemma}
\begin{proof}
We define $X_n$ to be the last $\tau$-renewal up to $n$. It is sufficient to show that there exists $c_{14}>0$ such that for large $n$, for any $0\leq m \leq n$
\begin{equation}\label{tiedown}
  \bP(2n \in \tau \mid X_n = m) \leq c_{14} \bP(2n \in\tau) \, .
\end{equation}
To prove this, we write
\begin{align}\label{lasthit}
\bP(2n \in \tau \mid X_n = m)  = \sum_{j=1}^n  \bP(\tau_1 = j + n-m | \tau_1 \geq n-m)\bP(n-j \in \tau) .
\end{align}
We split this sum into $j\leq n/2$ and $j>n/2$.

For $j\leq n/2$, we use that $\bP(k\in\tau)$ is regularly varying and $n-j\geq n/2$, to bound the corresponding part of the sum in \eqref{lasthit} by
\[\sup_{ k \geq n/2} \bP(k\in\tau)  \times \sum_{j=1}^n \bP(\tau_1 = j + n-m | \tau_1 \geq n-m)  \leq c_{15} \bP(2n\in\tau) \, .\]

For $j> n/2$, we use that for $n\geq j> n/2$ and $n\geq m\geq 0$,
\[
  \bP(\tau_1=j+n-m \mid \tau_1\geq n-m) \leq c_{16} \bP(\tau_1=n \mid \tau_1\geq n-m) \leq c_{16} \bP(\tau_1=n \mid \tau_1\geq n)
  \]
to bound the corresponding part of the sum in \eqref{lasthit} by
\[ 
c_{16}  \frac{\bP(\tau_1=n)}{\bP(\tau_1\geq n)}U_n \leq 
\begin{cases}
  c_{17} \bP(2n \in\tau) \quad \text{if } \ga =0 ; \\
c_{18} n^{-1}U_n \leq c_{19}  \bP(2n\in\tau) \quad \text{if } \ga>0 .
\end{cases}
\]
Here for $\ga=0$ we used \eqref{alpha0}, and for $\ga>0$ we used the regular variation of $\bP(n\in\tau)$. This completes the proof of \eqref{tiedown}.
\end{proof}

\end{appendix}

\bibliographystyle{plain}

\end{document}